\documentclass{birkjour}
\newtheorem{theorem}{Theorem}[section]

\newtheorem{lemma}[theorem]{Lemma}
\newtheorem{proposition}[theorem]{Proposition}
\theoremstyle{definition}
\newtheorem{definition}[theorem]{Definition}
\newtheorem{remark}[theorem]{Remark}
\newtheorem{example}[theorem]{Example}
\numberwithin{equation}{section}

\begin{document}

%
%
%
%
%
%
%
%
%

\title[New characterizations of  weak CMP inverses ]
 {New characterizations of  weak CMP inverses }

\author[Shuxian Xu]{Shuxian Xu}

\address{%
	School of Mathematics\\        
	Southeast University\\
	Nanjing 210096\\
	China}

\email{shuxianxu215@163.com}

\author{Jianlong Chen}
\address{%
	School of Mathematics\\        
	Southeast University\\
	Nanjing 210096\\
	China}
\email{jlchen@seu.edu.cn}


\date{September 9, 2025}

\begin{abstract}
In 2025, Mosi\'{c} defined the weak CMP inverse utilizing a minimal rank weak Drazin inverse  instead of the Drazin inverse. The weak CMP inverse is a new wider
class of generalized inverses, of which the CMP and MPCEP inverse
are particular cases. In this paper, we provide several expressions, along with a number of new characterizations and properties for the weak CMP inverse.  
Moreover, we investigate the relationships between the weak CMP inverse and some well-known generalized inverses, such as the Moore-Penrose inverse and weak MPD inverse.
Finally, we show that the weak CMP inverse, weak MPD inverse and weak DMP inverse are all strong Bott-Duffin $(e,f)$-inverses.
\end{abstract}

\maketitle

\section{Introduction}

Let $\mathbb{C}^{m\times n}$ be the set of $m\times n$ complex matrices. For  $A\in\mathbb{C}^{m\times n}$, the symbols ${\rm{rank}}(A)$, $A^*$ and $\mathcal{R}(A)$  denote the rank, the conjugate transpose and the range space of $A$, respectively. The index of $A\in\mathbb{C}^{n\times n}$, denoted by $\operatorname{ind}(A)$, is the smallest nonnegative integer $k$ for which $\operatorname{rank}(A^k)=\operatorname{rank}(A^{k+1})$. 

The Moore-Penrose inverse \cite{MP}  and the Drazin inverse \cite{D1958} both originated in the 1950s and are the most widely used generalized inverses. The Moore-Penrose inverse of \( A \in \mathbb{C}^{m \times n} \) is unique \( A^{\dagger} = X \in \mathbb{C}^{n \times m} \) satisfying \( AXA = A \), \( XAX = X \), \( (AX)^{*} = AX \) and \( (XA)^{*} = XA \) \cite{MP}. 
The Drazin inverse of \( A \in \mathbb{C}^{n \times n} \) with  \( \text{ind}(A) = k \), is unique \( A^D = X \in \mathbb{C}^{n \times n} \) such that \( XAX = X \), \( AX = XA \) and \( A^{k + 1}X = A^k \) \cite{D1958}. In particular, for \( \text{ind}(A) = 1 \), \( A^D \) reduces to the group inverse \( A^\# \).

The core inverse for a square matrix $A$ with index $1$ was introduced by Baksalary and
Trenkler \cite{BT} by three equations, denoted by $A^{\scriptsize\textcircled{\#}}$, and can be expressed as $A^{\scriptsize\textcircled{\#}}=A^{\#}AA^{\dag}$. As an extension of the core inverse, Malik and Thome \cite{DMP}  defined the DMP inverse for square matrices of arbitrary indexes using the Drazin inverse  and the Moore-Penrose inverse. For $A \in \mathbb{C}^{n\times n}$, the DMP inverse of $A$ is defined to be the matrix  $A^{D}AA^{\dag}$ and denoted by $A^{D,\dag}$,
Dually, the MPD inverse of \( A \) is given by \( A^{\dagger,D} = A^\dagger AA^{D} \).

The CMP inverse of \( A \in \mathbb{C}^{n \times n} \)  was introduced by Mehdipour and Salemi \cite{CMP} using the core part \( AA^D A \)  and the Moore-Penrose inverse $A^\dagger$ of \( A \), and is defined as \( A^{c,\dagger} = A^\dagger AA^D AA^\dagger \). Moreover, the CMP inverse of a square matrix is extended to a rectangular matrix in \cite{RCMP}, and to Hilbert space operators in \cite{HCMP, CMPH}. More details about CMP inverse are available in \cite{Zuo, F1, F2, Ma, CMP1, XSZ}.

Weak Drazin inverses and minimal rank weak Drazin inverses of complex matrices were proposed by 
Campbell and Meyer \cite{Campbell1978} in 1978. They are easier to
compute than Drazin inverses and can be used as a replacement for Drazin
inverses in many applications.
\begin{definition}\cite{Campbell1978}
	For $A \in \mathbb{C}^{n \times n}$, $X \in \mathbb{C}^{n \times n}$ is called a weak Drazin inverse of  $A$ if $XA^{k+1}=A^{k}$ for some  $k\in \mathbb{N}$. Furthermore,
	$X$ is called  a minimal rank weak Drazin inverse of  $A$ if $XA^{k+1}=A^{k}$ and rank$(X)=$ rank$(A^{D})$. Dually, a minimal rank right weak Drazin inverse of $A$  is a solution to
	$A^{k+1}X=A^{k}$ and rank$(X)=$ rank$(A^{D})$.	
\end{definition}

It was proved in \cite{WC2023} that many generalized inverses, such as the core
inverse, the DMP inverse, the weak group inverse \cite{Wang} and the weak core inverse \cite{F2, Zhou}
are special cases of minimal rank weak Drazin inverses. 

Replacing the Drazin inverse $A^{D}$ with a minimal rank weak Drazin inverse
of $A$,	weaker versions of MPD and DMP
inverses were presented in \cite{WDMP}.
Let \( A \in \mathbb{C}^{n \times n} \). If \( X \) is an arbitrary but fixed minimal rank weak Drazin inverse of \( A \), the weak MPD inverse \cite{WDMP} of \( A \) is expressed as $A^{w,\dagger,D} = A^\dagger XA$.
For an arbitrary but fixed minimal rank right weak Drazin inverse \( Z \) of \( A \), the weak DMP inverse \cite{WDMP} of \( A \)  is given by $A^{w,D,\dagger} = AZA^\dagger$.

Inspired by the fact that minimal rank weak Drazin inverses are easier to compute than Drazin inverses, and building on a series of studies on the CMP inverse as well as the concepts of weak DMP and MPD inverses, Mosi\'{c} \cite{WCMP} defined the weak CMP inverse using a minimal rank weak Drazin inverse instead of the Drazin inverse. This new class of generalized inverses extends the well-known CMP and MPCEP inverse \cite{MPCEP}.

\begin{definition}\cite{WCMP}
	Let \( A \in \mathbb{C}^{n \times n} \) and \( \text{ind}(A) = k \).
	For an arbitrary but fixed minimal rank weak Drazin inverse \( X \) of \( A \), the weak CMP inverse of \( A \) is given as
	\[
	A^{w,c,\dagger} = A^\dagger A X A A^\dagger.
	\]
\end{definition}

Motivated by previous work, in this paper, we present several expressions,  some new characterizations and properties for the weak CMP inverse. Besides, we explore the relationships between the weak CMP inverse and some existing generalized inverses, including the Moore-Penrose inverse and weak MPD inverse. Then, for a minimal rank weak Drazin inverse \( X \) of \( A \), we prove that $A^{\dag}AXAA^{\dag}=A^{D}=A^{D, \dag}=A^{\dag,D}$  is equivalent to $A^{\dag}AXA=AXAA^{\dag}$, and also equivalent to $A$ being left $k$-EP and $XAA^{\dag}=A^{D}$.
Finally, it is shown that the weak CMP inverse, weak MPD inverse and weak DMP inverse are all strong Bott-Duffin $(e,f)$-inverses.

	\section{Characterizations and  expressions for weak CMP inverse}

Recall the Hartwig-Spindelb$\ddot{\rm o}$ck decomposition \cite[Corollary 6]{HS}, which plays an important role in the following proof. For any $A \in \mathbb{C}^{n \times n}$ with rank $r$, there exists a unitary matrix $U \in \mathbb{C}^{n \times n}$ such that
\begin{equation}\label{AHS}
	A=U\left[ \begin{matrix}
		\varSigma K&		\varSigma L\\
		0&		0\\
	\end{matrix} \right]U^{*},
\end{equation}
where  $\varSigma =\text{diag}(\sigma_1I_{r_1},\sigma_2I_{r_2},\ldots,\sigma_tI_{r_t})$ is the diagonal matrix of singular values of $A$, $\sigma_1 > \sigma_2 > \cdots > \sigma_t > 0, r_1 + r_2 + \cdots + r_t = r$, 
and $K \in \mathbb{C}^{r \times r}$ and $L \in \mathbb{C}^{r \times (n-r)}$ satisfy 
$KK^{*}+LL^{*}=I_{r}.$ 

Let the decomposition of $A$ be of the form in equation \eqref{AHS}. Then, it follows from \cite{BT} that 
\begin{equation}\label{AMPHS}
	A^{\dag}=U\left[ \begin{matrix}
		K^{*}\varSigma^{-1}	&	0	\\
		L^{*}\varSigma^{-1}		&0		\\
	\end{matrix} \right] U^{*},	
\end{equation} 
and it was proved in \cite[Corollary 3.3]{WC2023} that, $X$ is a minimal rank weak Drazin inverse of $A$ if and only if
\begin{equation}\label{AXHS}
	X = U\begin{bmatrix}
		Z & W \\
		0 & 0
	\end{bmatrix}U^{*}
\end{equation}
for some minimal rank weak Drazin inverse $Z$ of $\Sigma K$ and some $W$ satisfying $\mathcal{R}(W)\subseteq\mathcal{R}(Z)$.

Based on the above discussion, we obtain the following expression for the weak CMP inverse of $A$ in terms of the Hartwig-Spindelb$\ddot{\rm o}$ck decomposition. 

\begin{theorem}\label{key}
	Let $A\in \mathbb{C}^{n\times n}$ be given by \eqref{AHS}. If $X$ is an arbitrary but fixed minimal rank weak Drazin inverse of $A$, then 
	
	\begin{equation}\label{ACMP}
		A^{w, c,\dag}= A^{\dag}AXAA^{\dag}=U\left[ \begin{matrix}
			K^{*}KZ	&	0	\\
			L^{*}KZ		& 0		\\
		\end{matrix} \right] U^{*}
	\end{equation} 
	for some minimal rank weak Drazin inverse $Z$ of $\Sigma K$.
	
\end{theorem}	

\begin{proof}
	It can be proved by substituting equations \eqref{AHS}, \eqref{AMPHS} and \eqref{AXHS} into $A^{\dag}AXAA^{\dag}$.
\end{proof}

The following result gives some properties of the weak CMP inverse of $A\in \mathbb{C}^{n\times n}$.	
\begin{proposition}
	Let $A\in \mathbb{C}^{n\times n}$ be given by \eqref{AHS} and $X$  be a minimal rank weak
	Drazin inverse of $A$. If $Y=A^{\dag}AXAA^{\dag}$, then the following statements hold:
	\begin{itemize}
		\item [{\rm (1)}] $Y=0$ if and only if $A$ is nilpotent;
		\item [{\rm (2)}]  $Y=A$ if and only if $A^{\dag}=A$;
		\item [{\rm (3)}]  $Y=A^{*}$ if and only if $A^{\dag}=A^{*}$ and ${\text ind}(A)=1$.
	\end{itemize}
\end{proposition}
\begin{proof}
	
	(1) If	$Y=0$, by \eqref{ACMP}, $K^{*}KZ=0$ and $L^{*}KZ=0$. Pre-multiplying  by $K$ and $L$ on these two equations respectively, and using $KK^{*}+LL^{*}=I_{r}$, we obtain $KZ=0$. Since $Z$ is a minimal rank weak Drazin inverse  of $\Sigma K$, it follows that $Z=Z(\Sigma K)Z=0$. Thus, $\Sigma K$ is nilpotent and hence $A$ is nilpotent. Conversely, if $A$ is nilpotent, then $X=0$, and hence $Y=0$.
	
	(2) If $Y=A$, then by \eqref{AHS} and \eqref{ACMP}, we have $K^{*}KZ=\Sigma K$ and $\Sigma L=0$. Since  $\Sigma$ is nonsingular and $KK^{*}+LL^{*}=I_{r}$, it follows that $L=0$ and $KK^{*}=I_{r}$. Thus, $\Sigma K$ is nonsingular and $K^{*}\varSigma^{-1}=K^{*}K(\varSigma K)^{-1}=K^{*}KZ=\Sigma K$. Therefore, $A^{\dag}=A$. Conversely, if $A^{\dag}=A$, then by \eqref{AHS} and \eqref{AMPHS}, $K^{*}\varSigma^{-1}=\Sigma K$, $L=0$ and $KK^{*}=I_{r}$. Thus, $\Sigma K$ is nonsingular and $Z=\varSigma K$.
	
	(3) If $Y=A^{*}$, then by \eqref{AHS} and \eqref{ACMP}, we have $K^{*}KZ=K^{*}\Sigma$ and $L^{*}KZ=L^{*}\Sigma$. Pre-multiplying  by $K$ and $L$ on these two equations respectively, and using $KK^{*}+LL^{*}=I_{r}$, we obtain $KZ=\Sigma$. Thus, $K$ is nonsingular, and hence $\Sigma=K(\varSigma K)^{-1}=\varSigma^{-1}$, which implies that $\varSigma=I_{r}$. Therefore, $Y=A^{*}$ if and only if $K$ is nonsingular and $\varSigma=I_{r}$. By \cite[Corollary 6]{HS}, $K$ is invertible if and only if ${\text {ind}}(A)=1$. And by \cite[Lemma 1]{BT}, $\varSigma=I_{r}$ if and only if $A^{\dag}=A^{*}$. Therefore, $Y=A^{*}$ if and only if $A^{\dag}=A^{*}$ and 
	${\text {ind}}(A)=1$.

\end{proof}

\begin{lemma}\label{A+}
	Let $A\in \mathbb{C}^{n\times n}$ be given by \eqref{AHS}. Then $X_1$ is a minimal rank weak
	Drazin inverse of $A^{\dag}$ if and only if 
	$X_1 = U\begin{bmatrix}
		Z_1 & 0 \\
		W_1 & 0
	\end{bmatrix}U^{*}$
	for some minimal rank weak Drazin inverse $Z_1$ of $K^{*}\Sigma^{-1} $ and some $W_1$ satisfying $\mathcal{R}({W_1}^{*})\subseteq\mathcal{R}({Z_1}^{*})$. 
\end{lemma}

\begin{proof}
	It can be derived by \eqref{AMPHS} and \cite[Proposition 3.2]{WC2023}.
\end{proof}

Applying Lemma \ref{A+}, we obtain the  relationship between the weak CMP inverses of $A$ and $A^{\dag}$.  
\begin{proposition}
	Let $A\in \mathbb{C}^{n\times n}$ be given by \eqref{AHS},  $X$ and $X_1$  be  minimal rank weak
	Drazin inverses of $A$ and $A^{\dag}$, respectively. Then $Y=A^{\dag}AXAA^{\dag}$ and $Y_1=AA^{\dag}X_{1}A^{\dag}A$ are weak CMP inverses of $A$ and $A^{\dag}$, respectively, and the following statements are equivalent:
	\begin{itemize}
		\item [{\rm (1)}] $Y^{\dag}=Y_1$;
		\item [{\rm (2)}]  $(KZ)^{\dag}=Z_1K^{*}$, where $Z$ and $Z_1$  are  minimal rank weak
		Drazin inverses of $\Sigma K$ and $K^{*}\Sigma^{-1}$, respectively.
	\end{itemize}
\end{proposition}

\begin{proof}
	First, by Theorem \ref{key}, the weak CMP inverse of $A$ has the form as 
	\begin{align*}
		Y= A^{\dag}AXAA^{\dag}=U\left[ \begin{matrix}
			K^{*}KZ	&	0	\\
			L^{*}KZ		& 0		\\
		\end{matrix} \right] U^{*}	
	\end{align*}
	for some minimal rank weak Drazin inverse $Z$ of $\Sigma K$.
	A direct calculation shows that
	\begin{align*}
		Y^{\dag}= U\left[ \begin{matrix}
			(KZ)^{\dag}K	&	(KZ)^{\dag}L	\\
			0	& 0		\\
		\end{matrix} \right] U^{*}.
	\end{align*} 	 
	In addition, by \eqref{AMPHS} and lemma \ref{A+}, 	
	\begin{align*}
		Y_1=AA^{\dag}X_{1}A^{\dag}A
		=U\left[ \begin{matrix}
			Z_{1}K^{*}K	&	Z_{1}K^{*}L	\\
			0	& 0		\\
		\end{matrix} \right] U^{*}
	\end{align*}
	for some minimal rank weak Drazin inverse $Z_1$ of $K^{*}\Sigma^{-1} $.
	Thus, $Y^{\dag}=Y_1$ if and only if $(KZ)^{\dag}K=Z_{1}K^{*}K$ and $(KZ)^{\dag}L=Z_{1}K^{*}L$. Post-multiplying by $K^{*}$ and $L^{*}$  on  these two equations respectively, and using $KK^{*}+LL^{*}=I_{r}$, we obtain $(KZ)^{\dag}=Z_1K^{*}$. 
	Therefore, $Y^{\dag}=Y_1$ if and only if $(KZ)^{\dag}=Z_1K^{*}$.
	
\end{proof}

In \cite[Theorem 3.2]{WCMP}, Mosi\'{c} gave a novel  expression for the weak CMP inverse $Y$ of $A$ based on projections $C=I-AY$ and $D=I-YA$ by using core-EP decomposition. Now, we give a simple proof of Theorem 3.2 in \cite{WCMP}.

\begin{proposition}{\rm{\cite{WCMP}}}\label{CD}
	Let \( A \in \mathbb{C}^{n \times n} \), \( \text{ind}(A) = k \) and \( X \) be a minimal rank weak Drazin inverse of \( A \). 
	For \( Y = A^{\dag}AXAA^{\dag} \), \( C = I - AY \) and \( D = I - YA \), matrices \( A \pm C \) are invertible and
	\[Y = (I - D)(A \pm C)^{-1}(I - C). \]
\end{proposition}
\begin{proof}
	First, we have $I - C=AY=AXAA^{\dag}$.
	Since $(XAA^{\dag})A^{k+1}=XA^{k+1}=A^{k}$ and 
	$A(XAA^{\dag})(XAA^{\dag})=A(XAA^{\dag})AX^{2}AA^{\dag}=AXAX^{2}AA^{\dag}=XAA^{\dag}$, it follows that $XAA^{\dag}$ is a minimal rank weak Drazin inverse of \( A \). Then by \cite[Theorem 3.10]{XSX}, we have $A \pm C$ are invertible with
	\begin{align*}
		&(A+C)^{-1}=\left(A+(I-AXAA^{\dag})\right)^{-1}\\
		&=XAA^{\dag}+(I-XAA^{\dag}A)\sum\limits_{i=0}\limits^{k-1}{(-A)^{i}}
		=XAA^{\dag}+(I-XA)\sum\limits_{i=0}\limits^{k-1}{(-A)^{i}},\\
		&(A-C)^{-1}=\left(A-(I-AXAA^{\dag})\right)^{-1}\\
		&=XAA^{\dag}-(I-XAA^{\dag}A)\sum\limits_{i=0}\limits^{k-1}{A^{i}}
		=XAA^{\dag}-(I-XA)\sum\limits_{i=0}\limits^{k-1}{A^{i}}.
	\end{align*}
	Thus,	
	\begin{align*}
		&(I - D)(A + C)^{-1}(I - C)\\
		&=A^{\dag}AXA\left(XAA^{\dag}+(I-XA)\sum\limits_{i=0}\limits^{k-1}{(-A)^{i}}\right)AXAA^{\dag}\\
		&=A^{\dag}AXAXAA^{\dag}AXAA^{\dag}=A^{\dag}AXAA^{\dag}=Y.
	\end{align*}
	Similarly, it can be verified that 	$(I - D)(A - C)^{-1}(I - C)=Y$.
\end{proof}

\begin{remark}
	Recall the Jacobson's lemma \cite{Jacobson}: Let $A, B \in \mathbb{C}^{n \times n}$. If $I_{n}-AB$ is invertible, then so is $I_{n}-BA$ and $(I_{n}-BA)^{-1}=I_{n}+B(I_{n}-AB)^{-1}A$. In Proposition \ref{CD}, since \( A \pm C \) are invertible, by Jacobson's lemma, \( A \pm D \) are also invertible, with
	\begin{align*}
		(A+D)^{-1}&=\left(A+(I-YA)\right)^{-1}\\
		&=I+(Y-I)\left(A+(I-AY)\right)^{-1}A\\
		&=I+(Y-I)(A+C)^{-1}A,\\
		(A-D)^{-1}&=\left(A-(I-YA)\right)^{-1}\\
		&=-I+(Y+I)\left(A-(I-AY)\right)^{-1}A\\
		&=-I+(Y+I)(A-C)^{-1}A.
	\end{align*}
\end{remark}
Let \( A \in \mathbb{C}^{n \times n} \) and \( \text{ind}(A) = k \). For a minimal rank weak Drazin inverse \( X \) of \( A \), it was proved in \cite[Lemma 2.1]{WCMP} that the weak CMP inverse
\( Y = A^{\dagger}AXAA^{\dagger} \) satisfies
$AY = P_{\mathcal{R}(A^k), \mathcal{N}(XAA^{\dagger})}$ and
$YA = P_{\mathcal{R}(A^{\dagger}A^k), \mathcal{N}(XA)}$. Now,
the following result further reveals the relationship among them.

\begin{proposition}
	Let \( A \in \mathbb{C}^{n \times n} \) and \( \text{ind}(A) = k \). For a minimal rank weak Drazin inverse \( X \) of \( A \),  the following statements are equivalent:
	\begin{itemize}
		\item [{\rm (1)}] \( Y = A^{\dagger}AXAA^{\dagger} \);
		\item [{\rm (2)}] $AY = P_{\mathcal{R}(A^k), \mathcal{N}(XAA^{\dagger})}$,
		$YA = P_{\mathcal{R}(A^{\dagger}A^k), \mathcal{N}(XA)}$, $YAY=Y$;
		\item [{\rm (3)}]  $AY = P_{\mathcal{R}(A^k), \mathcal{N}(XAA^{\dagger})}$,
		$YA = P_{\mathcal{R}(A^{\dagger}A^k), \mathcal{N}(XA)}$,  ${\text{rank}}(Y)={\text{rank}}(A^k)$.
	\end{itemize}
\end{proposition}

\begin{proof}
	(1)$\Rightarrow$(2) follows by \cite[Lemma 2.1]{WCMP}.
	
	(2)$\Rightarrow$(3). Since $AY = P_{\mathcal{R}(A^k), \mathcal{N}(XAA^{\dagger})}$ and $YAY=Y$, it follows that  ${\text{rank}}(Y)={\text{rank}}(AY)={\text{dim}}\mathcal{R}(A^k)={\text{rank}}(A^k)$.
	
	(3)$\Rightarrow$(1).  Since $\mathcal{N}(AY)=\mathcal{N}(XAA^{\dagger})$, it follows that $\mathcal{R}((XAA^{\dagger})^*)=\mathcal{R}(Y^{*}A^{*})\subseteq \mathcal{R}(Y^{*})$. And by ${\text{rank}}(Y)={\text{rank}}(A^k)={\text{rank}}(XA)={\text{rank}}(XAA^{\dag})$, we obtain $\mathcal{R}(AA^{\dagger}X^*)= \mathcal{R}(Y^{*})$. Also, 
	$$YA = P_{\mathcal{R}(A^{\dagger}A^k), \mathcal{N}(XA)}=(A^{\dagger}AXAA^{\dagger})A=A^{\dagger}AXA.$$
	Then by \cite[Theorem 2.5]{WCMP}, the proof is completed.
\end{proof}

\section{Relations with other generalized inverses}

In 1971, Rao and Mitra  introduced a class of generalized inverses for square complex matrices in \cite{Rao}: Given \( A \in \mathbb{C}^{n \times n} \), a matrix \( X \in \mathbb{C}^{n \times n} \) is called a \(\chi\)-inverse of \( A \) if it satisfies \( AXA = A \) and \( \mathcal{R}(X) \subseteq \mathcal{R}(A) \). Both the group
inverse and  core inverse are \(\chi\)-inverses. In addition,
it was proved in \cite[Proposition 2.1]{WC2025} that $X\in \mathbb{C}^{n \times n}$ is a $\chi$-inverse of $A$ if and only if ${\text {ind}}(A) \leq 1$ and $X$  is a minimal rank weak
Drazin inverse of $A$.

In the following result, we consider the equivalent conditions when the weak CMP inverse coincides with the Moore-Penrose inverse.
\begin{proposition}\label{withMP}
	Let $A\in \mathbb{C}^{n\times n}$ be given by \eqref{AHS} and $X$  be a minimal rank weak
	Drazin inverse of $A$. If $Y=A^{\dag}AXAA^{\dag}$, then the following statements are equivalent:
	\begin{itemize}
		\item [{\rm (1)}] $Y=A^{\dag}$;
		\item [{\rm (2)}]  ${\text {ind}}(A) \leq 1$;
		\item [{\rm (3)}]  $A=AA^{D}A$;
		\item [{\rm (4)}]  $A=XA^{2}$;
		\item [{\rm (5)}]  $X$ is a $\chi$-inverse of $A$.
	\end{itemize}
\end{proposition}	

\begin{proof}
	(1)$\Rightarrow$(2): Suppose  $Y=A^{\dag}$. Then by  \eqref{AMPHS} and  \eqref{ACMP}, we have $K^{*}\varSigma^{-1}=K^{*}KZ$ and $L^{*}\varSigma^{-1}=L^{*}KZ$. Pre-multiplying  by $K$ and $L$ on these two equations respectively, and using $KK^{*}+LL^{*}=I_{r}$, we obtain $\varSigma^{-1}=KZ$, and hence $\varSigma KZ=I_r$. Thus, $\varSigma K$ is nonsingular and ${\rm rank}(A)=r$. Furthermore, since
	\begin{align*}
		A^2 = U\begin{pmatrix}
			\Sigma K & 0 \\
			0 & 0
		\end{pmatrix}\begin{pmatrix}
			\Sigma K & \Sigma L \\
			0 & 0
		\end{pmatrix}U^* = U\begin{pmatrix}
			\Sigma K & 0 \\
			0 & I_{n - r}
		\end{pmatrix}\begin{pmatrix}
			\Sigma K & \Sigma L \\
			0 & 0
		\end{pmatrix}U^*,
	\end{align*} 
	it follows that ${\rm rank}(A)={\rm rank}(A^2)=r$, which implies that ${\text {ind}}(A)\leq 1$.
	
	(2)$\Rightarrow$(1): Since ${\text {ind}}(A) \leq 1$, it follows that $A=XA^{2}$ and $A=A^{2}A^{D}$. Thus,
	\begin{align*}
		Y&=A^{\dag}AXAA^{\dag}
		=A^{\dag}AX(A^{2}A^{D})A^{\dag}\\
		&=A^{\dag}A(XA^{2})A^{D}A^{\dag}
		=A^{\dag}AAA^{D}A^{\dag}=A^{\dag}AA^{\dag}=A^{\dag}.
	\end{align*}
	
	(2)$\Leftrightarrow$(3)$\Leftrightarrow$(4) is clear.
	
	(2)$\Leftrightarrow$(5) follows by \cite[Proposition 2.1]{WC2025}.
\end{proof}

\begin{remark}
	Note that, if $A$ is an EP matrix (i.e. $A^{\dag}=A^{\#}$), then ${\text {ind}}(A) \leq 1$, and by Proposition \ref{withMP}, the weak CMP inverse $A^{w,c,\dag}=A^{\dag}=A^{\#}=A^{\scriptsize\textcircled{\#}}$, which implies that $A^{w,c,\dag}$
	is also EP.  However, \cite[Theorem 2.5]{CMP} is sufficient to illustrate that $A^{w,c,\dag}$
	being  EP may not imply $A$ being  EP.
\end{remark}

Analogously as Proposition \ref{withMP}, we provide the necessary and sufficient conditions for the weak MPD inverse coincides with the Moore-Penrose inverse.
\begin{proposition}\label{wMPD-MP}
	Let $A\in \mathbb{C}^{n\times n}$ be given by \eqref{AHS} and $X$  be a minimal rank weak
	Drazin inverse of $A$. If $Y=A^{\dag}XA$, then the following statements are equivalent:
	\begin{itemize}
		\item [{\rm (1)}] $Y=A^{\dag}$;
		\item [{\rm (2)}]  ${\text {ind}}(A)=1$ and $L=0$.	
	\end{itemize}
	
\end{proposition}	

\begin{proof}
	(1)$\Rightarrow$(2):
	Let $A$ be given by \eqref{AHS}. Applying \eqref{AMPHS} and \eqref{AXHS}, the weak MPD inverse
	\begin{equation}\label{wMPD}
		Y=A^{\dag}XA=U\left[\begin{matrix}
			K^{*}\varSigma^{-1}Z\varSigma K	&	K^{*}\varSigma^{-1}Z\varSigma L	\\
			L^{*}\varSigma^{-1}Z\varSigma K	& L^{*}\varSigma^{-1}Z\varSigma L		\\
		\end{matrix}\right] U^{*}.
	\end{equation} 
	A direct computation shows that $Y=A^{\dag}$ if and only if the following conditions hold.
	\begin{equation}\label{01}
		K^{*}\varSigma^{-1}Z\varSigma K=K^{*}\varSigma^{-1},
	\end{equation}
	\begin{equation}\label{02}
		L^{*}\varSigma^{-1}Z\varSigma K=L^{*}\varSigma^{-1},
	\end{equation}
	\begin{equation}\label{03}
		K^{*}\varSigma^{-1}Z\varSigma L=0,
	\end{equation}
	\begin{equation}\label{04}
		L^{*}\varSigma^{-1}Z\varSigma L=0.
	\end{equation}	
	Pre-multiplying \eqref{01} and \eqref{02} by $K$ and $L$  respectively, and using $KK^{*}+LL^{*}=I_{r}$, we obtain $\varSigma^{-1}Z\varSigma K=\varSigma^{-1}$. 
	Thus, $Z\varSigma K=I_{r}$ and $Z=(\varSigma K)^{-1}$, which implies that $K$ is invertible.
	By \cite[Corollary 6]{HS}, $K$ is invertible if and only if ${\text {ind}}(A)=1$. Similarly, pre-multiplying \eqref{03} and \eqref{04} by $K$ and $L$  respectively, we get $\varSigma^{-1}Z\varSigma L=0$, which implies that $L=0$.
	
	(2)$\Rightarrow$(1) is clear.
\end{proof}

Now, we  present the equivalent conditions under which the weak CMP inverse coincides with the weak MPD inverse.	
\begin{proposition}\label{wMPD-wCMP}
	Let $A\in \mathbb{C}^{n\times n}$ be given by \eqref{AHS} and $X$  be a minimal rank weak
	Drazin inverse of $A$. Then $X$ is expressed as \eqref{AXHS},
	and the following statements are equivalent:
	\begin{itemize}
		\item [{\rm (1)}] $A^{\dag}AXAA^{\dag}=A^{\dag}XA$;
		\item [{\rm (2)}]  $Z=(\varSigma K)^{D}$ and $(\varSigma K)^{D}\varSigma L=0$.	
	\end{itemize}
	
\end{proposition}		

\begin{proof}
	(1)$\Rightarrow$(2): Applying \eqref{ACMP} and \eqref{wMPD}, $A^{\dag}AXAA^{\dag}=A^{\dag}XA$ if and only if the following conditions hold.
	\begin{equation}\label{1-1}
		K^{*}\varSigma^{-1}Z\varSigma K=K^{*}KZ,
	\end{equation}
	\begin{equation}\label{1-2}
		L^{*}\varSigma^{-1}Z\varSigma K=L^{*}KZ,
	\end{equation}
	\begin{equation}\label{1-3}
		K^{*}\varSigma^{-1}Z\varSigma L=0,
	\end{equation}
	\begin{equation}\label{1-4}
		L^{*}\varSigma^{-1}Z\varSigma L=0.
	\end{equation}
	Pre-multiplying \eqref{1-1} and \eqref{1-2} by $K$ and $L$  respectively, and using $KK^{*}+LL^{*}=I_{r}$, we obtain $\varSigma^{-1}Z\varSigma K=KZ$. Thus, $Z\varSigma K=\varSigma KZ$, which implies that $Z=(\varSigma K)^{D}$.
	Similarly, pre-multiplying \eqref{1-3} and \eqref{1-4} by $K$ and $L$  respectively, we get $\varSigma^{-1}Z\varSigma L=0$, which implies that $(\varSigma K)^{D}\varSigma L=Z\varSigma L=0$.
	
	(2)$\Rightarrow$(1): Suppose $Z=(\varSigma K)^{D}$ and $(\varSigma K)^{D}\varSigma L=0$. Then a direct calculation shows that \eqref{1-1}-\eqref{1-4} hold.
\end{proof}

Recall that $A\in \mathbb{C}^{n\times n}$ is called a Core-EP matrix \cite{CMP} if it satisfies $A^{\dag}AA^{D}A=AA^{D}AA^{\dag}$. Motivated by the Core-EP matrix, if $X$  is a minimal rank weak Drazin inverse of $A$, we investigate
equivalent conditions for the equality $A^{\dag}AXA=AXAA^{\dag}$.

\begin{lemma}\label{wC-EP}
	Let $A\in \mathbb{C}^{n\times n}$ be given by \eqref{AHS} and $X$  be a minimal rank weak
	Drazin inverse of $A$. Then $X$ is expressed as \eqref{AXHS},
	and the following statements are equivalent:
	\begin{itemize}
		\item [{\rm (1)}] $A^{\dag}AXA=AXAA^{\dag}$;
		\item [{\rm (2)}]  $K^{*}KZ=Z$, $L^{*}KZ=0$, $Z\varSigma L=0$
		and $Z=(\varSigma K)^{D}$.	
	\end{itemize}
\end{lemma}
\begin{proof}
	(1)$\Rightarrow$(2): Let $A$ be given by \eqref{AHS}. Then applying \eqref{AMPHS} and \eqref{AXHS}, we obtain
	\begin{align*}
		A^{\dag}AXA=U\left[\begin{matrix}
			K^{*}KZ\varSigma K	&	K^{*}KZ\varSigma L	\\
			L^{*}KZ\varSigma K		& L^{*}KZ\varSigma L		\\
		\end{matrix}\right] U^{*},
	\end{align*} 
	\begin{align*}
		AXAA^{\dag}=U\left[\begin{matrix}
			\varSigma KZ & 0	\\
			0		& 0		\\
		\end{matrix}\right] U^{*}.
	\end{align*} 
	Thus, $A^{\dag}AXA=AXAA^{\dag}$ if and only if the following conditions hold.
	\begin{equation}\label{2-1}
		K^{*}KZ\varSigma K=\varSigma KZ,
	\end{equation}
	\begin{equation}\label{2-2}
		K^{*}KZ\varSigma L=0,
	\end{equation}
	\begin{equation}\label{2-3}
		L^{*}KZ\varSigma K=0,
	\end{equation}
	\begin{equation}\label{2-4}
		L^{*}KZ\varSigma L=0.
	\end{equation}
	Post-multiplying \eqref{2-1} and \eqref{2-3} by $Z$ and using $Z(\varSigma K) Z=Z$, we obtain $K^{*}KZ=\varSigma KZ^{2}=Z$ and $L^{*}KZ=0$. Substituting $K^{*}KZ=Z$ into \eqref{2-1}, we get 
	$Z\varSigma K=\varSigma KZ$, which implies that $Z=(\varSigma K)^{D}$. Substituting $K^{*}KZ=Z$ into \eqref{2-2} yields $Z\varSigma L=K^{*}KZ\varSigma L=0$.
	
	(2)$\Rightarrow$(1):  Suppose $K^{*}KZ=Z$, $L^{*}KZ=0$, $Z\varSigma L=0$
	and $Z=(\varSigma K)^{D}$.	Then \eqref{2-2}-\eqref{2-4} can be obtained directly, and  $K^{*}KZ\varSigma K=Z\varSigma K=(\varSigma K)^{D}\varSigma K=\varSigma K(\varSigma K)^{D}=\varSigma KZ$.
\end{proof}	

\begin{remark}
	Let $A\in \mathbb{C}^{n\times n}$ be given by \eqref{AHS}. If $X$  is a minimal rank weak
	Drazin inverse of $A$, then by Proposition \ref{wMPD-wCMP} and Lemma \ref{wC-EP},  we know that $A^{\dag}AXA=AXAA^{\dag}$ can imply $A^{\dag}AXAA^{\dag}=A^{\dag}XA$. 
	However, the following example shows that $A^{\dag}AXAA^{\dag}=A^{\dag}XA$ may not imply $A^{\dag}AXA=AXAA^{\dag}$.
	
\end{remark}

\begin{example}
	Let	$A=\left(\begin{matrix}
		1  &   1 &    0  &   1\\
		0  &   0   &  0  &  -1\\
		0   &  0   &  0  &   0\\
		0  &   0   &  0  &   0
	\end{matrix}
	\right)$. Then  
	$A^{\dag}=\left(\begin{matrix}
		1/2  &   1/2 &    0  &   0\\
		1/2  &   1/2   &  0  &  0\\
		0   &  0   &  0  &   0\\
		0  &   -1   &  0  &   0
	\end{matrix}
	\right),$
	and 
	$X=\left(\begin{matrix}
		1  &   1 &    0  &   1\\
		0  &   0   &  0  &  0\\
		0   &  0   &  0  &   0\\
		0  &   0   &  0  &   0
	\end{matrix}
	\right)$ is a minimal rank  weak Drazin inverse of $A$.
	A direct calculation shows that 
	$$A^{\dag}AXAA^{\dag}
	=\left(\begin{matrix}
		1/2  &   1/2 &    0  &   0\\
		1/2  &   1/2   &  0  &  0\\
		0   &  0   &  0  &   0\\
		0  &   0   &  0  &   0
	\end{matrix}
	\right)
	=A^{\dag}XA.$$
	However, $A^{\dag}AXA\neq AXAA^{\dag}$.

\end{example}

\begin{remark}
	Let $A\in \mathbb{C}^{n\times n}$ be given by \eqref{AHS} and  $X$  be a minimal rank weak
	Drazin inverse of $A$. By \cite[Lemma 3.2]{CMP}, $A$ is a Core-EP matrix if and only if 
	$K^{*}K(\varSigma K)^{D}=(\varSigma K)^{D}$, $L^{*}K(\varSigma K)^{D}=0$ and 
	$(\varSigma K)^{D}\varSigma L=0$.
	Then by Lemma \ref{wC-EP},	we obtain $A^{\dag}AXA=AXAA^{\dag}$ implies that $A$ is Core-EP matrix. So it
	is natural to consider whether they are equivalent.
	But the following example shows that the converse is not true in general.
\end{remark}	

\begin{example}
	Let	$A=\left(\begin{matrix}
		1  &   0 &    0  &   0\\
		0  &   0   &  0  &  1\\
		0   &  0   &  0  &   0\\
		0  &   0   &  0  &   0
	\end{matrix}
	\right)$ with $\text{ind}(A)=2$. By computation, we obtain that  $A^{\dag}=\left(\begin{matrix}
		1  &   0 &    0  &   0\\
		0  &   0   &  0  &  0\\
		0   &  0   &  0  &   0\\
		0  &   1   &  0  &   0
	\end{matrix}
	\right)$ and $A^{D}=\left(\begin{matrix}
		1  &   0 &    0  &   0\\
		0  &   0   &  0  &  0\\
		0   &  0   &  0  &   0\\
		0  &   0   &  0  &   0
	\end{matrix}
	\right)$. Also, $X=\left(\begin{matrix}
		1  &   -1 &    0  &   1\\
		0  &   0   &  0  &  0\\
		0   &  0   &  0  &   0\\
		0  &   0   &  0  &   0
	\end{matrix}
	\right)$ is a minimal rank weak Drazin inverse of $A$.
	A direct calculation shows that
	$$A^{\dag}AA^{D}A=\left(\begin{matrix}
		1  &   0 &    0  &   0\\
		0  &   0   &  0  &  0\\
		0   &  0   &  0  &   0\\
		0  &   0   &  0  &   0
	\end{matrix}
	\right)=AA^{D}AA^{\dag},$$
	which implies that $A$ is Core-EP matrix. However, $A^{\dag}AXA \neq AXAA^{\dag}$.

\end{example}

Applying Lemma \ref{wC-EP}, we obtain the following.
\begin{theorem}\label{main}
	Let $A\in \mathbb{C}^{n\times n}$  and $X$  be a minimal rank weak
	Drazin inverse of $A$. Then the following statements are equivalent:
	\begin{itemize}
		\item [{\rm (1)}] $A^{\dag}AXA=AXAA^{\dag}$;
		\item [{\rm (2)}]  $A^{\dag}AXAA^{\dag}=A^{D, \dag}$ and $A^{\dag}AXAA^{\dag}=A^{\dag,D}$;
		\item [{\rm (3)}]  	$A^{\dag}AXAA^{\dag}=A^{D}=A^{D, \dag}=A^{\dag,D}$;
		\item [{\rm (4)}] $XAA^{\dag}=A^{D}$ and $A^{\dag}A^{k+1}=A^{k}$.
	\end{itemize}
\end{theorem}

\begin{proof}
	Let $A\in \mathbb{C}^{n\times n}$ be given by \eqref{AHS}. It follows from \cite{DMP} that
	\begin{equation}\label{AD}
		A^{D}=U\left[ \begin{matrix}
			(\varSigma K)^{D}	&	((\varSigma K)^{D})^{2}\varSigma L	\\
			0		&0		\\
		\end{matrix} \right] U^{*},
	\end{equation}
	\begin{equation}\label{ADMP}
		A^{D, \dag}=U\left[ \begin{matrix}
			(\varSigma K)^{D}	&	0	\\
			0		&0		\\
		\end{matrix} \right] U^{*},
	\end{equation}
	\begin{equation}\label{AMPD}
		A^{\dag,D}=U\left[ \begin{matrix}
			K^{*}K(\varSigma K)^{D}	&	K^{*}K((\varSigma K)^{D})^{2}\varSigma L	\\
			L^{*}K(\varSigma K)^{D}		& L^{*}K((\varSigma K)^{D})^{2}\varSigma L		\\
		\end{matrix} \right] U^{*}.
	\end{equation}
	
	(1)$\Rightarrow$(2): Suppose $A^{\dag}AXA=AXAA^{\dag}$. Then $K^{*}KZ=Z$, $L^{*}KZ=0$, $Z\varSigma L=0$
	and $Z=(\varSigma K)^{D}$ by Lemma \ref{wC-EP}. Thus, we obtain $K^{*}KZ=(\varSigma K)^{D}$ and $L^{*}KZ=0$, which implies that  $A^{\dag}AXAA^{\dag}=A^{D, \dag}$ by \eqref{ACMP} and \eqref{ADMP}. Moreover, since $K^{*}K(\varSigma K)^{D}=K^{*}KZ$, $L^{*}K(\varSigma K)^{D}=L^{*}KZ$, 
	$K^{*}K((\varSigma K)^{D})^{2}\varSigma L=0$ and $L^{*}K((\varSigma K)^{D})^{2}\varSigma L=0$, it follows that $A^{\dag,D}=A^{\dag}AXAA^{\dag}$  by \eqref{ACMP} and \eqref{AMPD}.

	(2)$\Rightarrow$(3): Suppose $A^{\dag}AXAA^{\dag}=A^{D, \dag}$ and $A^{\dag}AXAA^{\dag}=A^{\dag,D}$. Then by $A^{\dag}AXAA^{\dag}=A^{D, \dag}$ and \eqref{ACMP} and 
	\eqref{ADMP}, we obtain $K^{*}KZ=(\varSigma K)^{D}$ and $L^{*}KZ=0$. Since $A^{\dag}AXAA^{\dag}=A^{\dag,D}$, it follows that $K^{*}K((\varSigma K)^{D})^{2}\varSigma L=0$ and $L^{*}K((\varSigma K)^{D})^{2}\varSigma L=0$ by \eqref{ACMP} and \eqref{AMPD}. Pre-multiplying these two equations by $K$ and $L$  respectively, and using $KK^{*}+LL^{*}=I_{r}$, we obtain $K((\varSigma K)^{D})^{2}\varSigma L=0$. Thus, $(\varSigma K)^{D}\varSigma L=\varSigma K((\varSigma K)^{D})^{2}\varSigma L=0$. Therefore, by \eqref{ACMP} and \eqref{AD}, $A^{\dag}AXAA^{\dag}=A^{D}$.
	
	(3)$\Rightarrow$(4): Since $A^{\dag}AXAA^{\dag}=A^{D}$, it follows that
	$$XAA^{\dag}=A^{D}AXAA^{\dag}=(A^{\dag}AXAA^{\dag})AXAA^{\dag}=A^{\dag}AXAA^{\dag}=A^{D},$$
	and 
	$$A^{\dag}A^{k+1}=A^{\dag}AA^{k}=A^{\dag}AXA^{k+1}=A^{\dag}AXAA^{\dag}A^{k+1}=A^{D}A^{k+1}   =A^{k}.$$
	
	(4)$\Rightarrow$(1): Suppose $XAA^{\dag}=A^{D}$ and $A^{\dag}A^{k+1}=A^{k}$. Then $AXAA^{\dag}=AA^{D}$, and $A^{\dag}AXA=A^{\dag}A^{k+1}X^{k+1}A=A^{k}X^{k+1}A=XA=XAA^{\dag}A=A^{D}A$. Thus, $A^{\dag}AXA=AXAA^{\dag}$.
	
	\begin{remark}
		Let  $A$ be a complex matrix with index $k$. Recall that  $A$ is called $k$-EP \cite{EPmatrix} if it satisfies $A^{\dagger}A^{k}=A^{k}A^{\dagger}$.
		It was proved in \cite[Corollary 3.18]{F1} that $A^{c,\dag}=A^{D}=A^{D, \dag}=A^{\dag,D}$  if and only if $A$ is $k$-EP matrix.
		As  one side case of  $k$-EP matrix, it was proved in \cite{WuChen} that $A$ is left $k$-EP (or left power-EP)  if and only if $A^{\dagger}A^{k+1}=A^{k}$. Then by Theorem \ref{main}, $A^{\dag}AXAA^{\dag}=A^{D}=A^{D, \dag}=A^{\dag,D}$ if and  only if  $A$ is left $k$-EP and $XAA^{\dag}=A^{D}$ .
	\end{remark}	
\end{proof}

In \cite{D2023}, Drazin  showed that in fact the DMP and CMP inverses are both special cases of the class of Bott-Duffin $(e,f)$-inverses\textemdash strong Bott-Duffin $(e,f)$-inverses.

\begin{definition}\cite[Definition 2.4]{D2023}
	Given any semigroup \( S \) and any two idempotents \( e, f \in S \), we call an element \( a \in S \) \emph{strongly Bott--Duffin $(e, f)$-invertible} if there exists \( y \in S \) such that \( yay = y \), \( ya = e \) and \( ay = f \).
\end{definition}
If any such \( y \) exists, then it is unique and  called  the strong Bott--Duffin $(e, f)$-inverse of \( a \). 

Inspired by this, we investigate the relationship between weak CMP inverse (resp. weak MPD inverse, weak DMP inverse) and strong Bott-Duffin $(e,f)$-inverse.

\begin{proposition}
	Let $A\in \mathbb{C}^{n\times n}$  and $X$  be a minimal rank weak
	Drazin inverse of $A$. If $Y=A^{\dag}AXAA^{\dag}$ , then $Y$ is a strong Bott-Duffin $(A^{\dag}AXA, AXAA^{\dag})$-inverse of $A$.
\end{proposition}
\begin{proof}
	It follows by the proof of \cite[Theorem 2.3]{D2023}.
\end{proof}

\begin{proposition}\label{wMPD-BD}
	Let $A\in \mathbb{C}^{n\times n}$  and $X$  be a minimal rank weak
	Drazin inverse of $A$. If $Y=A^{\dag}XA$ , then $Y$ is a strong Bott-Duffin $(A^{\dag}XA^{2}, AA^{\dag}XA)$-inverse of $A$.
\end{proposition}
\begin{proof}
	Let $E=A^{\dag}XA^{2}$ and $F=AA^{\dag}XA$. Since
	\begin{align*}
	E^{2}&=A^{\dag}XA^{2}A^{\dag}XA^{2}=A^{\dag}XA^{2}A^{\dag}AX^{2}A^{2}\\
	&=A^{\dag}XA^{2}X^{2}A^{2}=A^{\dag}XAXA^{2}=A^{\dag}XA^{2}
	\end{align*}
	and
	 \begin{align*}
	 	F^{2}&=AA^{\dag}XAAA^{\dag}XA=AA^{\dag}XAAA^{\dag}AX^{2}A\\
	 	&=AA^{\dag}XAAX^{2}A=AA^{\dag}XAXA=AA^{\dag}XA,
	 \end{align*}
	it follows that $E$ and $F$ are both  idempotents. 
	Moreover, a direct calculation shows that $YA=A^{\dag}XA^{2}=E$ and $AY=AA^{\dag}XA=F$. 
	Also, $YAY=Y$ follows by the definition of weak MPD inverse. Thus, $Y$ is a strong Bott-Duffin $(A^{\dag}XA^{2}, AA^{\dag}XA)$-inverse of $A$.
	
\end{proof}

Similarly, we obtain the following relationship between   weak DMP inverse and strong Bott-Duffin $(e,f)$-inverse.

\begin{proposition}
	Let $A\in \mathbb{C}^{n\times n}$  and $X$  be a minimal rank right weak
	Drazin inverse of $A$. If $Y=AXA^{\dag}$ , then $Y$ is a strong Bott-Duffin $(AXA^{\dag}A, A^{2}XA^{\dag})$-inverse of $A$.
\end{proposition}
\begin{proof}
	The proof  is similar to Proposition \ref{wMPD-BD}.
\end{proof}

\section*{Acknowledgment}
This research was supported by the National Natural Science Foundation of China (No. 12171083) and the Jiangsu Provincial Scientific Research Center of Applied Mathematics under Grant No. BK20233002. We thank Dr. Yukun Zhou for meaningful suggestions.


\begin{thebibliography}{1}


\bibitem{BT} O.M. Baksalary, G. Trenkler, \textit{Core inverse of matrices.} Linear Multilinear Algebra \textbf{58}(6) (2010), 681--697.


\bibitem{Campbell1978}   S.L. Campbell, 
C.D. Meyer, \textit{Weak Drazin inverses.} Linear Algebra Appl. \textbf{20}(2) (1978), 167--178. 

\bibitem{Zuo} Y. Chen, K.Z. Zuo, Q.W. Wang, Z.M. Fu, \textit{Further characterizations of the CMP inverse of matrices.} Linear Multilinear Algebra \textbf{71}(1) (2023), 100--119. 


\bibitem{D1958} M.P. Drazin, \textit{Pseudo-inverses in associative rings and semigroups.} Amer. Math. Monthly \textbf{65} (1958), 506--514.

\bibitem{D2023}	M.P. Drazin, \textit{Intermixing pairs of generalized inverses.}
Linear Multilinear Algebra \textbf{71}(8) (2023), 1397--1406.


\bibitem{F2} D.E. Ferreyra, F.E. Levis, A.N. Priori, N. Thome, \textit{The weak core inverse.} Aequationes Math. \textbf{95}(2) (2021), 351--373.

\bibitem{F1} D.E. Ferreyra, F.E. Levis, N. Thome, \textit{Characterizations of $k$-commutative equalities for some outer generalized inverses.} Linear Multilinear Algebra \textbf{68}(1) (2020), 177--192.

\bibitem{F2} D.E. Ferreyra, F.E. Levis, N. Thome, \textit{Maximal classes of matrices determining generalized inverses.} Appl. Math. Comput. \textbf{333}(2018), 42--52. 


\bibitem{HS} R.E. Hartwig, K. Spindelb$\ddot{\rm o}$ck, \textit{Matrices for which $A^{*}$ and $A^{\dag}$ commute} Linear Multilinear Algebra \textbf{14}(3) (1983), 241--256.



\bibitem{Jacobson} N. Jacobson, \textit{The radical and semi-simplicity for arbitrary rings.}
Amer. J. Math. \textbf{67} (1945), 300--320.

\bibitem{Ma} H.F. Ma, \textit{Characterizations and representations for the CMP inverse and its application.} Linear Multilinear Algebra \textbf{70}(20) (2022), 5157--5172. 


\bibitem{DMP} S.B. Malik, N. Thome, \textit{On a new generalized inverse for matrices of an arbitrary index.} Appl. Math. Comput. \textbf{226} (2014), 575--580.

\bibitem{EPmatrix} S.B. Malik, L. Rueda, N. Thome, \textit{The class of $m$-EP and $m$-normal matrices.} Linear Multilinear Algebra \textbf{64}(11) (2016), 2119--2132.

\bibitem{CMP} M. Mehdipour, A. Salemi, \textit{On a new generalized inverse of matrices.} Linear Multilinear Algebra \textbf{66}(5) (2018), 1046--1053.

\bibitem{MPCEP} D. Mosi\'{c}, I.I. Kyrchei, P.S. Stanimirovi\'{c}, \textit{Representations and properties for the MPCEP inverse.} J. Appl. Math. Comput. \textbf{67}(1-2) (2021), 101--130.	


\bibitem{RCMP} D. Mosi\'{c}, \textit{The CMP inverse for rectangular matrices.} Aequationes Math. \textbf{92}(4)	(2018), 649--659.

\bibitem{CMP1} D. Mosi\'{c}, \textit{Integral and limit representations of the CMP inverse.} Miskolc Math. Notes \textbf{21}(2) (2020), 983--992. 

\bibitem{HCMP} D. Mosi\'{c}, \textit {Maximal classes of operators determining some weighted generalized inverses.} Linear Multilinear Algebra \textbf{68}(11) (2020), 2201--2220.

\bibitem{WDMP} D. Mosi\'{c}, \textit{Weak MPD and DMP inverses.} J. Math. Anal. Appl. (2024), 128653.

\bibitem{WCMP} D. Mosi\'{c}, \textit{Weak CMP inverses.} Aequationes Math. DOI:10.1007/s00010-025-01167-4.

\bibitem{CMPH} D. Mosi\'{c}, M.Z. Kolund\v{z}ija, \textit{Weighted CMP inverse of an operator between Hilbert spaces.} Rev. R. Acad. Cienc. Exactas F\'{i}s. Nat. Ser. A Mat. RACSAM. \textbf{113} (2019), 2155--2173.


\bibitem{MP} R. Penrose, \textit{A generalized inverse for matrices.} Proc. Cambridge Philos. Soc. \textbf{51} (1955), 406--413.

\bibitem{Rao} C.R. Rao, S.K. Mitra, \textit{Generalized Inverses of Matrices and its Applications.} Wiley, New York, 1971.

\bibitem{Wang} H.X. Wang, J.L. Chen, \textit{Weak group inverse.} Open Math. \textbf{16}(2018), 1218--1232.

\bibitem{WuChen} C. Wu, J.L. Chen, \textit{Left and right power-EP matrices.} Linear Multilinear Algebra \textbf{71}(15) (2023), 2527--2542.

\bibitem{WC2023} C. Wu, J.L. Chen, \textit{Minimal rank weak Drazin inverses: a class of outer inverses with prescribed range.} Electron. J. Linear Algebra \textbf{39} (2023), 1--16.

\bibitem{WC2025} C. Wu, J.L. Chen, Y.Y. Ke, \textit{Jacobson's lemma and Cline's formula for minimal rank weak Drazin inverses.}
Filomat \textbf{39}(8) (2025), 2557--2564.


\bibitem{XSX} S.X. Xu, J.L. Chen, C. Wu, \textit{Minimal weak Drazin inverses in semigroups and rings.} Mediterr. J. Math. \textbf{21}(3) (2024), Paper No. 119, 15 pp.

\bibitem{XSZ} S.Z. Xu, J.L. Chen, D. Mosi\'{c}, \textit{New characterizations of the CMP inverse of matrices.} Linear Multilinear Algebra \textbf{68}(4) (2020), 790--804.

\bibitem{Zhou} Y.K. Zhou, J.L. Chen, \textit{Weak core inverses and pseudo core inverses in a ring with involution.} Linear Multilinear Algebra \textbf{70}(21) (2022), 6876--6890.


\end{thebibliography}
\end{document}